\documentclass[11pt]{amsart}
\usepackage{amsmath,amssymb,latexsym,soul,cite,amsthm,color,enumitem,graphicx,mathtools,microtype}
\usepackage[colorlinks=true,urlcolor=blue,citecolor=blue,linkcolor=blue,linktocpage,pdfpagelabels,bookmarksnumbered,bookmarksopen]{hyperref}
\usepackage[english]{babel}
\usepackage[all,cmtip]{xy}
\usepackage[left=2.5cm,right=2.5cm,top=2.5cm,bottom=2.5cm]{geometry}

\numberwithin{equation}{section}

\newtheorem{theorem}{Theorem}[section]
\theoremstyle{plain}
\newtheorem{lemma}[theorem]{Lemma}
\theoremstyle{plain}
\newtheorem{proposition}[theorem]{Proposition}
\theoremstyle{plain}

\newtheorem{definition}[theorem]{Definition}
\theoremstyle{definition}
\newtheorem{remark}[theorem]{Remark}
\newtheorem{example}[theorem]{Example}

\newcommand{\N}{{\mathbb N}}

\newcommand{\R}{{\mathbb R}}
\newcommand{\eps}{\varepsilon}
\newcommand{\beq}{\begin{equation}}
\newcommand{\eeq}{\end{equation}}
\renewcommand{\le}{\leqslant}
\renewcommand{\ge}{\geqslant}

\newcommand{\w}{W^{s,p}_0(\Omega)}
\newcommand{\fpl}{(-\Delta)_p^s\,}
\newcommand{\ds}{{\rm d}_\Omega^s}
\newcommand{\cs}{C_s^0(\overline\Omega)}

\newenvironment{enumroman}{\begin{enumerate}

}{\end{enumerate}}

\title[Superlinear fractional $p$-Laplacian]{Multiple solutions for superlinear fractional $p$-Laplacian equations}

\author[A.\ Iannizzotto, V.\ Staicu, V.\ Vespri]{Antonio Iannizzotto, Vasile Staicu, Vincenzo Vespri}

\address[A.\ Iannizzotto, corresponding author]{Department of Mathematics and Computer Science
\newline\indent
University of Cagliari
\newline\indent
Via Ospedale 72, 09124 Cagliari, Italy}
\email{antonio.iannizzotto@unica.it}

\address[V.\ Staicu]{CIDMA - Center for Research and Development in Mathematics and Applications
\newline\indent
Department of Mathematics
\newline\indent
University of Aveiro
\newline\indent
3810-193 Aveiro, Portugal}
\email{vasile@ua.pt}

\address[V.\ Vespri]{Department of Mathematics and Computer Science ''U.\ Dini''
\newline\indent
University of Florence
\newline\indent
Viale Morgagni 67/A, 50134 Firenze, Italy}
\email{vincenzo.vespri@unifi.it}


\begin{document}

\begin{abstract}
We study a Dirichlet problem driven by the (degenerate or singular) fractional $p$-Laplacian and involving a $(p-1)$-superlinear reaction at infinity, not necessarily satisfying the Ambrosetti-Rabinowitz condition. Using critical point theory, truncation, and Morse theory, we prove the existence of at least three nontrivial solutions to the problem.
\end{abstract}

\maketitle

\begin{center}
Version of \today\
\end{center}

\section{Introduction and main result}\label{sec1}

\noindent
In this paper we study a nonlinear partial differential equation of fractional order, coupled with a nonlocal Dirichlet condition:
\beq\label{dir}
\begin{cases}
\fpl u = f(x,u) & \text{in $\Omega$} \\
u = 0 & \text{in $\Omega^c$.}
\end{cases}
\eeq
The data of \eqref{dir} are as follows: $\Omega\subset\R^N$ ($N\ge 2$) is a bounded domain with $C^{1,1}$-smooth boundary $\partial\Omega$, $p>1$ and $s\in(0,1)$ are real s.t.\ $ps<N$, and $f:\Omega\times\R\to\R$ is a Carath\'eodory mapping with $(p-1)$-superlinear growth. The leading operator is the $s$-fractional $p$-Laplacian, heuristically defined for all $u:\R^N\to\R$ smooth enough by
\[\fpl u(x) = 2\lim_{\eps\to 0^+}\int_{B^c_\eps(x)}\frac{|u(x)-u(y)|^{p-2}(u(x)-u(y))}{|x-y|^{N+ps}}\,dy.\]
Nonlinear elliptic equations with superlinear reactions represent a challenging subject in nonlinear analysis, because the corresponding energy functional is unbounded from below, hence one cannot find solutions by the direct method of the calculus of variations. Minimization is then replaced by convenient {\em min-max} schemes, according to the general definition of \cite{R}, which often lead to multiplicity results.
\vskip2pt
\noindent
The study of superlinear equations started with \cite{AR}, where existence of {\em two} nontrivial solutions was proved via the mountain pass theorem for a semilinear elliptic equation whose reaction is superlinear at infinity and satisfies the Ambrosetti-Rabinowitz condition, while infinitely many solutions appear if the reaction is odd, thanks to the symmetric mountain pass theorem. Later, in \cite{W} the existence of {\em three} nontrivial solutions (with sign information) was shown under the same hypotheses via the application of Morse theory for $C^2$-functionals and the separate study of truncated reactions. No symmetry conditions were involved.
\vskip2pt
\noindent
Since then, a vast literature has developed on superlinear problems, using a number of variational techniques to prove multiplicity results. For short, we only list some results related to fractional order equations: \cite{A,GZZ,CW} deal with odd superlinear reactions, \cite{DI,IL,MP,XZR,ZF,ZCC} with superlinear reactions at infinity without symmetry conditions (on the contrary, asymmetric behaviors are sometimes exploited to prove multiplicity), and \cite{AI,FMBS} consider problems with critical growth, which are a special case of superlinear problems. In particular, \cite[Theorem 4.3]{DI} is a formal analogue of the fine result of \cite{W} for the fractional Laplacian ($p=2$).
\vskip2pt
\noindent
The aim of the present paper is to extend the result of \cite{W} (and henceforth of \cite{DI}) to the case of the fractional $p$-Laplacian, with any $p>1$. Before presenting the precise statement of our main result, let us briefly revise the theoretical framework of variational methods applied to fractional elliptic equations (see \cite{MBRS} for a general introduction). When compared to classical elliptic boundary value problems, fractional problems present a very similar variational structure, with the fractional Sobolev space $\w$ replacing $W^{1,p}_0(\Omega)$ (see \cite{DNPV,L}). The most significant differences occur in regularity of solutions (especially boundary regularity) and maximum-comparison principles:
\begin{itemize}[leftmargin=1cm]
\item[$(a)$] While fairly regular in the interior of the domain, weak solutions to fractional Dirichlet problems may fail to be differentiable at the boundary (see \cite{ROS} for the linear case), hence boundary regularity for this kind of problems cannot aim at a $C^{1,\alpha}$-regularity like in the case of the $p$-Laplacian (see \cite{L1}). Nevertheless, such property is effectively replaced by a form of fine (or weighted) boundary regularity, namely, whenever $u$ solves \eqref{dir}, the quotient $u/\ds(x)$ (${\rm d}_\Omega$ denoting the distance from the boundary) is H\"older continuous up to the boundary. This result was proved in \cite{IM} for all $p>1$. See also the recent contributions \cite{CDI,DIV} for interior regularity.
\item[$(b)$] The strong maximum and comparison principles are, surprisingly, easier to obtain in the fractional case with respect to the classical $p$-Laplacian (see \cite{GV,V}). For instance, it is proved in \cite{IMP} that any non-negative, non-zero solution of \eqref{dir} lies in the interior of the positive order cone of a convenient weighted space $\cs$, incorporating a fractional version of Hopf's boundary point property. Note that no integral condition is imposed to control the behavior of the reaction near the origin, as was the case for the classical $p$-Laplacian, see \cite{V}. This, in turn, leads to the comparison between local minimizers of the energy functional in the topologies of $\w$ and $\cs$, which can be proved to coincide, see \cite{IM1,IMS}.
\end{itemize}
The issues above are especially relevant in the study of problem \eqref{dir} with $(p-1)$-superlinear reactions. In order to distinguish positive, negative, and general solutions we introduce several truncations of the reaction and we need to separate the critical points of the corresponding truncated energy functionals. This is possible through a delicate interplay between the $\w$ and $\cs$-topologies, respectively, and making essential use of the recent results on regularity and maximum principles for the fractional $p$-Laplacian.
\vskip2pt
\noindent
We will study problem \eqref{dir} under the following hypotheses (where we denote $p^*_s=Np/(N-ps)$ the fractional Sobolev exponent):
\begin{itemize}[leftmargin=.6cm]
\item[${\bf H.}$] $f:\Omega\times\R\to\R$ is a Carath\'eodory mapping, we set for all $(x,t)\in\Omega\times\R$
\[F(x,t) = \int_0^t f(x,\tau)\,d\tau,\]
and assume that the following conditions are satisfied:
\begin{enumroman}
\item\label{h1} there exist $c_0>0$, $r\in(p,p^*_s)$ s.t.\ for a.e.\ $x\in\Omega$ and all $t\in\R$
\[|f(x,t)| \le c_0(1+|t|^{r-1});\]
\item\label{h2} uniformly for a.e.\ $x\in\Omega$
\[\lim_{|t|\to\infty}\frac{F(x,t)}{|t|^p} = \infty;\]
\item\label{h3} there exists $q\in\big(\frac{N}{ps}(r-p),p^*_s\big)$ s.t.\ uniformly for a.e.\ $x\in\Omega$
\[\liminf_{|t|\to\infty}\frac{f(x,t)t-pF(x,t)}{|t|^q} > 0;\]
\item\label{h4} uniformly for a.e.\ $x\in\Omega$
\[\lim_{t\to 0}\frac{f(x,t)}{|t|^{p-2}t} = 0.\]
\end{enumroman}
\end{itemize}
Our result is the following:

\begin{theorem}\label{main}
Let ${\bf H}$ hold. Then, problem \eqref{dir} has at least three nontrivial solutions, $u_+>0$, $u_-<0$, and $\tilde u\neq 0$.
\end{theorem}

\noindent
For a precise definition of solution we refer to Section \ref{sec2} below. We note that hypotheses ${\bf H}$ \ref{h2} \ref{h4} agree with a $(p-1)$-superlinear behavior of $f$ both at infinity and at the origin, while ${\bf H}$ \ref{h3} is a generalization of the Ambrosetti-Rabinowitz condition, first introduced in \cite{CM} for semilinear elliptic equations. The latter is essential, along with the subcritical growth condition \ref{h1}, to ensure that the energy functional of problem \eqref{dir} satisfies the Cerami compactness condition at any level. Again, no symmetry condition is required.
\vskip2pt
\noindent
The proof of Theorem \ref{main} is naturally split in two parts. In the first part, we detect the constant sign solutions $u_\pm$, by applying the mountain pass theorem to convenient truncations of the energy functional: in this argument we make a major use of the aforementioned results from \cite{IMP,IM1}. In the second part, we seek a third solution $\tilde u\neq 0$ arguing by contradiction and computing all critical groups at $u_\pm$, at the origin, and at infinity, respectively, of the energy functional: this method is an adaptation of \cite{W} to the generalized $C^1$-Morse theory displayed in \cite{MMP}. We remark that this is the {\em first} result of this type for fractional $p$-laplacian equations including the singular case ($1<p<2$).

\begin{example}\label{ex1}
Fix $q\in(p,p^*_s)$ and set for all $t\in\R$
\[f(t) = |t|^{q-2}t,\]
then it is easily seen that the (autonomous) mapping $f\in C(\R)$ satisfies ${\bf H}$ (with $r=q$), and it also satisfies the classical Ambrosetti-Rabinowitz condition. Alternatively, we may set for all $t\in\R$
\[F(t) = \begin{cases}
0 & \text{if $t=0$} \\
\displaystyle\frac{|t|^q}{q}\ln|t| & \text{if $0<|t|<1$} \\
\displaystyle\frac{|t|^p}{q}\ln|t| & \text{if $|t|\ge 1$.}
\end{cases}\]
So $F\in C^1(\R)$ and $f=F'$ satisfies ${\bf H}$ (with any $r\in(q,p^*_s)$), but it does not satisfy the Ambrosetti-Rabinowitz condition. Both reactions considered here are odd, but one can easily think of different mappings without symmetries.
\end{example}

\noindent
The structure of the paper is the following: in Section \ref{sec2} we recall some preliminary results on the Dirichlet problem for the fractional $p$-Laplacian (Subsection \ref{ss21}) and general Morse theory (Subsection \ref{ss22}); in Section \ref{sec3} we deal with constant sign solutions; and in Section \ref{sec4} we prove existence of the third nontrivial solution.
\vskip4pt
\noindent
{\bf Notations.} Throughout the paper, for any $\Omega\subset\R^N$ we shall set $\Omega^c=\R^N\setminus\Omega$ and denote by $|\Omega|$ the $N$-dimensional Lebesgue measure of $\Omega$. For any two measurable functions $u,v:\Omega\to\R$, $u\le v$ in $\Omega$ will mean that $u(x)\le v(x)$ for a.e.\ $x\in\Omega$ (and similar expressions). The positive (resp., negative) part of $u$ is denoted $u^+$ (resp., $u^-$). If $X$ is an ordered Banach space, then $X_+$ will denote its non-negative order cone and $B_r(x)$, $\overline{B}_r(x)$, $\partial B_r(x)$ will be the open ball, the closed ball, and the sphere, respectively, centered at $x$ with radius $r>0$. For all $r\in[1,\infty]$, $\|\cdot\|_r$ denotes the standard norm of $L^r(\Omega)$ (or $L^r(\R^N)$, which will be clear from the context). Every function $u$ defined in $\Omega$ will be identified with its $0$-extension to $\R^N$. Moreover, $C$ will denote a positive constant (whose value may change case by case).

\section{Preliminaries}\label{sec2}

\noindent
For the reader's convenience, we collect here some known definitions and results that will be used in the forthcoming arguments.

\subsection{The Dirichlet problem for the fractional $p$-Laplacian}\label{ss21}

First we recall some notions on the Dirichlet problem \eqref{dir}, referring to \cite{DNPV,L,P} for a detailed account. Let $p$, $s$ be as in Section \ref{sec1}. For all open $\Omega\subseteq\R^N$ and all measurable $u:\Omega\to\R$ we define the Gagliardo seminorm
\[[u]_{s,p,\Omega} = \Big[\iint_{\Omega\times\Omega}\frac{|u(x)-u(y)|^p}{|x-y|^{N+ps}}\,dx\,dy\Big]^\frac{1}{p}.\]
We define the fractional Sobolev space
\[W^{s,p}(\Omega) = \big\{u\in L^p(\Omega):\,[u]_{s,p,\Omega}<\infty\big\}.\]
Let now $\Omega$ be as in Section \ref{sec1}. The space
\[\w = \big\{u\in W^{s,p}(\R^N):\,u=0 \ \text{in $\Omega^c$}\big\}\]
is a uniformly convex, separable Banach space with norm $\|u\|_{s,p}=[u]_{s,p,\R^N}$ and dual space $W^{-s,p'}(\Omega)=(\w)^*$. The space $\w$ is continuously embedded into $L^q(\Omega)$ for all $q\in[1,p^*_s]$ and compactly embedded into $L^q(\Omega)$ for all $q\in[1,p^*_s)$. We will also use the following topological vector space:
\[\widetilde{W}^{s,p}(\Omega) = \Big\{u\in L^p_{\rm loc}(\R^N):\,u\in W^{s,p}(U) \ \text{for some $U\Supset\Omega$,} \ \int_{\R^N}\frac{|u(x)|^{p-1}}{(1+|x|)^{N+ps}}\,dx<\infty\Big\}.\]
The $s$-fractional $p$-Laplacian can be defined as a continuous operator $\fpl:\widetilde{W}^{s,p}(\Omega)\to W^{-s,p'}(\Omega)$ by setting for all $u\in\widetilde{W}^{s,p}(\Omega)$, $\varphi\in\w$
\[\langle\fpl u,\varphi\rangle = \iint_{\R^N \times \R^N} \frac{|u(x)-u(y)|^{p-2} (u(x)-u(y)) (\varphi(x)-\varphi(y))}{|x-y|^{N+ps}}\,dx\,dy.\]
Such definition agrees with the one given in Section \ref{sec1} for $p\ge 2$ and $u$ smooth enough. The operator $\fpl$ is strictly $T$-monotone \cite[Proposition 2.1]{IMP}. We are more interested in the restricted operator to the subspace $\w\subset\widetilde{W}^{s,p}(\Omega)$:

\begin{proposition}\label{pro}
{\rm\cite[Lemma 2.1]{FI} \cite[Lemma 2.1]{IL}} The operator $\fpl:\w\to W^{-s,p'}(\Omega)$ is continuous, the gradient of $u\mapsto\|u\|_{s,p}^p/p$, and of type $(S)_+$, i.e., whenever $u_n\rightharpoonup u$ in $\w$ and
\[\limsup_n\,\langle\fpl u_n,u_n-u\rangle \le 0,\]
then $u_n\to u$ in $\w$. Finally, for all $u\in\w$
\[\|u^\pm\|_{s,p}^p \le \langle\fpl u,\pm u^\pm\rangle.\]
\end{proposition}

\noindent
We consider now problem \eqref{dir} under the following, more general hypothesis (equivalent to ${\bf H}$ \ref{h1}):
\begin{itemize}[leftmargin=.7cm]
\item[${\bf H}_0.$] $f:\Omega\times\R\to\R$ is a Carath\'eodory mapping, and there exist $c_0>0$, $r\in(p,p^*_s)$ s.t.\ for a.e.\ $x\in\Omega$ and all $t\in\R$
\[|f(x,t)| \le c_0(1+|t|^{r-1}).\]
\end{itemize}
We introduce the notions of weak (super, sub-) solutions:

\begin{definition}\label{sol}
We say that $u\in\widetilde{W}^{s,p}(\Omega)$ is a (weak) supersolution of \eqref{dir} if $u\ge 0$ in $\Omega^c$ and for all $\varphi\in\w_+$
\[\langle\fpl u,\varphi\rangle \ge \int_\Omega f(x,u)\varphi\,dx.\]
Similarly, $u\in\widetilde{W}^{s,p}(\Omega)$ is a subsolution of \eqref{dir} if $u\le 0$ in $\Omega^c$ and for all $\varphi\in\w_+$
\[\langle\fpl u,\varphi\rangle \le \int_\Omega f(x,u)\varphi\,dx.\]
Finally, we say that $u$ is a solution of \eqref{dir} if $u$ is both a supersolution and a subsolution.
\end{definition}

\noindent
Equivalently, $u\in\w$ solves \eqref{dir} if for all $\varphi\in\w$
\beq\label{weak}
\langle\fpl u,\varphi\rangle = \int_\Omega f(x,u)\varphi\,dx.
\eeq
The notions of super- and subsolutions given in Definition \ref{sol} above (borrowed from \cite{P}) are not the most general ones. Indeed, the space $\widetilde{W}^{s,p}(\Omega)$ may be replaced with suitable subspaces of $L^p_{\rm loc}(\R^N)$ (see for instance \cite{BGCV}).
\vskip2pt
\noindent
We next recall some qualitative properties of solutions, starting with a priori bounds:

\begin{proposition}\label{apb}
{\rm\cite[Proposition 2.3]{IM1}} Let ${\bf H}_0$ hold, $u\in\w$ be a solution of \eqref{dir}. Then, $u\in L^{\infty}(\Omega)$ with $\|u\|_{\infty} \le C$, for some $C=C(\|u\|_{s,p})>0$.
\end{proposition}

\noindent
While solutions of nonlocal problems are in general very regular in the interior, they fail to be smooth up to the boundary. An effective substitute for higher boundary regularity is a form of weighted H\"older continuity with weight
\[\ds(x) = {\rm dist}(x,\Omega^c)^s.\]
For all $\alpha\in[0,1)$ we define the space
\[C^\alpha_s(\overline\Omega) = \Big\{u\in C^0(\overline\Omega):\,\frac{u}{\ds} \ \text{has a $\alpha$-H\"older continuous extension to $\overline\Omega$}\Big\},\]
endowed with the norm $\|u\|_{\alpha,s}=\|u/\ds\|_{C^\alpha(\overline\Omega)}$ (for $\alpha=0$, the extension is only assumed to be continuous). In particular we will use the space $\cs$, whose positive order cone has a nonempty interior given by
\beq\label{int}
{\rm int}(\cs_+) = \Big\{u\in\cs:\,\inf_{x\in\Omega}\frac{u(x)}{\ds(x)}>0\Big\}.
\eeq
Combining Proposition \ref{apb} with \cite[Theorem 1.1]{IM}, we have the following global regularity result:

\begin{proposition}\label{reg}
Let ${\bf H}_0$ hold, $u \in \w$ be a solution of \eqref{dir}. Then, $u \in C_s^{\alpha}(\overline\Omega)$  for some $\alpha \in (0,s]$ depending only on $N,p,s,$ and $\Omega$, with $\|u\|_{\alpha,s}\le C(\|u\|_{s,p})$.
\end{proposition}

\noindent
The following strong maximum principle holds for continuous supersolutions:

\begin{proposition}\label{smp}
{\rm\cite[Theorem 2.6]{IMP}} Let $g \in C^0(\R) \cap BV_{\rm loc}(\R)$, $u\in\widetilde{W}^{s,p}(\Omega)\cap C^0(\overline\Omega)\setminus\{0\}$ s.t.\ 
\[\begin{cases}
\fpl u+g(u) \ge g(0) & \text{weakly in $\Omega$} \\
u \ge 0 & \text{in $\R^N$.}
\end{cases}\]
Then, 
\[\inf_{\Omega} \frac{u}{\ds} > 0.\]
In particular, if $u\in\cs$, then $u\in{\rm int}(\cs_+)$.
\end{proposition}

\noindent
As said in Section \ref{sec1}, our approach to problem \eqref{dir} is variational. We define an energy functional by setting for all $u\in\w$
\[\Phi(u) = \frac{\|u\|_{s,p}^p}{p}-\int_\Omega F(x,u)\,dx.\]
As proved in \cite{ILPS}, as soon as ${\bf H}_0$ is satisfied, $\Phi$ is sequentially weakly l.s.c.\ in $\w$ and $\Phi\in C^1(\w)$ with G\^ateaux derivative given for all $u,\varphi\in\w$ by
\[\langle\Phi'(u),\varphi\rangle = \langle\fpl u,\varphi\rangle-\int_\Omega f(x,u)\varphi\,dx.\]
In particular, $u$ satisfies \eqref{weak} iff $\Phi'(u)=0$ in $W^{-s,p'}(\Omega)$. Also, $\Phi$ enjoys the following bounded compactness condition, following from the $(S)_+$-property of $\fpl$ (Proposition \ref{pro}):

\begin{proposition}\label{bps}
{\rm\cite[Proposition 2.1]{ILPS}} Let ${\bf H}_0$ hold, and $(u_n)$ be a bounded sequence in $\w$ s.t.\ $\Phi'(u_n)\to 0$. Then, $(u_n)$ has a (strongly) convergent subsequence in $\w$.
\end{proposition}

\noindent
Another crucial tool on our study is the following equivalence principle for Sobolev and (weighted) H\"older local minima of $\Phi$:

\begin{proposition}\label{svh}
{\rm\cite[Theorem 3.1]{IM1}} Let ${\bf H}_0$ hold, $u\in\w$. Then, the following conditions are equivalent:
\begin{enumroman}
\item\label{svh1} there exists $\rho>0$ s.t.\ $\Phi(u+v)\ge\Phi(u)$ for all $v\in\w$, $\|v\|_{s,p}\le\rho$;
\item\label{svh2} there exists $\rho'>0$ s.t.\ $\Phi(u+v)\ge\Phi(u)$ for all $v\in\w\cap\cs$, $\|v\|_{0,s}\le\rho'$.
\end{enumroman}
\end{proposition}

\noindent
Finally we consider the eigenvalue problem for the fractional $p$-Laplacian with Dirichlet conditions:
\beq\label{evp}
\begin{cases}
\fpl u = \lambda|u|^{p-1}u & \text{in $\Omega$} \\
u = 0 & \text{in $\Omega^c$.}
\end{cases}
\eeq
For our needs, we only recall some properties of the principal eigenvalue:

\begin{proposition}\label{pev}
{\rm\cite[Theorems 4.1, 4.2]{FP}} The smallest eigenvalue of \eqref{evp} is
\[\lambda_1 = \min_{u\in\w\setminus\{0\}}\frac{\|u\|_{s,p}^p}{\|u\|_p^p} > 0,\]
it is simple, isolated, and attained at a unique positive eigenfunction $e_1\in{\rm int}(\cs_+)$ s.t.\ $\|e_1\|_p=1$.
\end{proposition}

\subsection{Some recalls of Morse theory}\label{ss22}

Here we collect some definitions from Morse theory for $C^1$-functionals defined in Banach spaces, mainly following \cite[Chapter 6]{MMP}. Such theory extends the classical Morse theory for $C^2$-functionals in Hilbert spaces (see \cite{C}), but it focuses on critical groups rather than on the Morse index. Let $X$ be a reflexive Banach space with ${\rm dim}(X)=\infty$.
\vskip2pt
\noindent
For every topological pair $(A,B)$, with $B\subseteq A\subseteq X$, and $k\in\N$, we denote by $H_k(A,B)$ the $k$-th singular homology group of the pair (in fact a real vector space), for every map $g:(A,B)\to (C,D)$ the homomorphism induced by $g$ is $g_*:H_k(A,B)\to H_k(C,D)$, and the boundary homomorphism is $\partial:H_k(A,B)\to H_{k-1}(A,\emptyset)$ \cite[Definition 6.9]{MMP}.
\vskip2pt
\noindent
We recall some properties of singular homology groups, that we will use in the sequel (for the notion of deformation retract see \cite[Definition 5.3]{MMP}):

\begin{proposition}\label{shg}
{\rm\cite[Corollary 6.15, Propositions 6.14, 6.24, 6.25]{MMP}} Let $(A,B)$ be a topological pair in $X$, $k\in\N$:
\begin{enumroman}
\item\label{shg1} if $B$ is a deformation retract of $A$, then $H_k(X,A)=H_k(X,B)$;
\item\label{shg2} if $A$ is a deformation retract of $X$, then $H_k(X,B)=H_k(A,B)$;
\item\label{shg3} if $A$ is contractible to a point $x\in X$, then $H_k(X,A)=H_k(X,\{x\})$;
\item\label{shg4} if $X$ is contractible to a point $x\in X$, then $H_k(X,\{x\})=0$;
\item\label{shg5} if $i:(A,B)\to(X,B)$, $j:(X,B)\to(X,A)$, and $h:(A,\emptyset)\to(A,B)$ are inclusion mappings, then the following group sequence is exact:
\[\ldots \ H_k(A,B) \ \xrightarrow{i_*} \ H_k(X,B) \ \xrightarrow{j_*} \ H_k(X,A) \ \xrightarrow{h_*\circ\partial} \ H_{k-1}(A,B) \ \ldots\]
\end{enumroman}
\end{proposition}

\noindent
Now consider a functional $\Phi\in C^1(X)$. We say that $\Phi$ satisfies the Cerami compactness condition ($(C)$-condition for short), if every sequence $(x_n)$ in $X$ s.t.\
\[|\Phi(x_n)| \le C, \ (1+\|x_n\|)\Phi'(x_n)\to 0\]
has a (strongly) convergent subsequence. We denote by $K(\Phi)$ the set of all critical points of $\Phi$ and set for all $c\in\R$
\[K_c(\Phi) = \big\{x\in K(\Phi):\,\Phi(x)=c\big\}.\]
We say that $x\in K(\Phi)$ is an isolated critical point, if there exists a neighborhood $U$ of $x$ s.t.\
\[K(\Phi)\cap U = \{x\}.\]
For any isolated critical point $x\in K_c(\Phi)$ and $U$ as above, and any $k\in\N$ we define the $k$-th critical group of $\Phi$ at $x$ by setting
\[C_k(\Phi,x) = H_k(\Phi^c\cap U,\Phi^c\cap U\setminus\{x\}),\]
where as usual $\Phi^c=\{\Phi\le c\}$. We note that the definition above is invariant with respect to $U$ by the excision property of homology groups \cite[Axiom 6, p.\ 143]{MMP}. We define the critical groups of $\Phi$ at infinity, by choosing a level $c<\inf_{x\in K(\Phi)}\Phi(x)$ and setting for all $k\in\N$
\[C_k(\Phi,\infty) = H_k(X,\Phi^c)\]
(this definition also is invariant with respect to $c$). Next we recall some properties of critical groups:

\begin{proposition}\label{cgp}
{\rm\cite[Example 6.45 $(a)$, Remark 6.58]{MMP}} Let $\Phi\in C^1(X)$ satisfy the $(C)$-condition:
\begin{enumroman}
\item\label{cgp1} if $x\in K(\Phi)$ is an isolated critical point and a strict local minimizer, then for all $k\in\N$
\[C_k(\Phi,x) = \delta_{k,0}\R;\]
\item\label{cgp2} if $b<c<a$ are real s.t.\ $K_c(\Phi)$ is finite and $K_{c'}(\Phi)=\emptyset$ for all $c'\in[b,a]\setminus\{c\}$, then for all $k\in\N$
\[H_k(\Phi^a,\Phi^b) = \bigoplus_{x\in K_c(\Phi)} C_k(\Phi,x);\]
\item\label{cgp3} {\rm(Poincar\'e-Hopf formula)} if $K(\Phi)$ is finite, all critical groups of $\Phi$ either at critical points or at infinity are finite-dimensional and vanish for $k$ big enough, then
\[\sum_{k=0}^\infty\sum_{x\in K(\Phi)}(-1)^k{\rm dim}(C_k(\Phi,x)) = \sum_{k=0}^\infty(-1)^k{\rm dim}(C_k(\Phi,\infty)).\]
\end{enumroman}
\end{proposition}

\noindent
Finally we recall a homotopy invariance property of critical groups at uniformly isolated critical points:

\begin{proposition}\label{inv}
{\rm\cite[Theorem 5.6]{C}} Let $(\Psi_\tau)_{\tau\in[0,1]}$ be a family of $C^1$-functionals satisfying the $(C)$-condition, $x\in X$, and $U\subset X$ be a neighborhood of $x$ s.t.\ $K(\Psi_\tau)\cap U=\{x\}$ for all $\tau\in[0,1]$. Then, for all $k\in\N$ and all $\tau,\tau'\in[0,1]$
\[C_k(\Psi_\tau,x) = C_k(\Psi_{\tau'},x).\]
\end{proposition}

\noindent
Note that reference \cite{C} is for $C^2$-functional on Hilbert spaces, but homotopy invariance can be proved similarly in our framework.

\section{Constant sign solutions}\label{sec3}

\noindent
Here we prove the existence of two constant sign solutions for \eqref{dir}, one positive and one negative. To achieve such result, we need to introduce some truncations of the energy functional $\Phi$. Note that by ${\bf H}$ \ref{h4} we have for a.e.\ $x\in\Omega$
\[f(x,0) = 0.\]
So we set for all $(x,t)\in\Omega\times\R$
\[f_\pm(x,t) = f(x,\pm t^\pm), \ F_\pm(x,t) = \int_0^t f_\pm(x,\tau)\,d\tau.\]
The mappings $f_\pm:\Omega\times\R\to\R$ satisfy ${\bf H}_0$. Also set for all $u\in\w$
\[\Phi_\pm(u) = \frac{\|u\|_{s,p}^p}{p}-\int_\Omega F_\pm(x,u)\,dx.\]
Reasoning as in Subsection \ref{ss21} we see that both functionals $\Phi_\pm\in C^1(\w)$ are sequentially weakly l.s.c., and enjoy the compactness property of Proposition \ref{bps}.

\begin{lemma}\label{cer}
Let ${\bf H}$ hold. Then, both $\Phi$ and $\Phi_\pm$ satisfy the $(C)$-condition.
\end{lemma}
\begin{proof}
We deal with $\Phi$. Let $(u_n)$ be a sequence in $\w$ s.t.\ $|\Phi(u_n)|\le C$ and $(1+\|u_n\|_{s,p})\Phi'(u_n)\to 0$ as $n\to\infty$. So we have for all $n\in\N$
\beq\label{cer1}
\Big|\|u_n\|_{s,p}^p-p\int_\Omega F(x,u_n)\,dx\Big| \le C,
\eeq
and there exists a sequence $(\eps_n)$ s.t.\ $\eps_n\to 0^+$ and for all $\varphi\in\w$
\beq\label{cer2}
\Big|\langle\fpl u_n,\varphi\rangle-\int_\Omega f(x,u_n)\varphi\,dx\Big| \le \frac{\eps_n\|\varphi\|_{s,p}}{1+\|u_n\|_{s,p}}.
\eeq
Set $\varphi=u_n$ in \eqref{cer2}:
\[-\|u_n\|_{s,p}^p+\int_\Omega f(x,u_n)u_n\,dx \le \eps_n.\]
Add \eqref{cer1}:
\[\int_\Omega\big(f(x,u_n)u_n-pF(x,u_n)\big)\,dx \le C.\]
By ${\bf H}$ \ref{h3} we can find $T,\beta>0$ s.t.\ for a.e.\ $x\in\Omega$ and all $|t|>T$
\[f(x,t)t-pF(x,t) \ge \beta|t|^q.\]
Besides, by ${\bf H}$ \ref{h1} we have for a.e.\ $x\in\Omega$ and all $t\in\R$
\beq\label{cer3}
|F(x,t)| \le C(1+|t|^r).
\eeq
By the previous relations we have for all $n\in\N$
\begin{align*}
C &\ge \int_{\{|u_n|>T\}}\beta|u_n|^q\,dx+\int_{\{|u_n|\le T\}}\big(f(x,u_n)u_n-pF(x,u_n)\big)\,dx \\
&\ge \beta\|u_n\|_q^q+\int_{\{|u_n|\le T\}}\big(f(x,u_n)u_n-pF(x,u_n)-\beta|u_n|^q\big)\,dx \\
&\ge \beta\|u_n\|_q^q-C(1+T^r+T^q).
\end{align*}
So $(u_n)$ is bounded in $L^q(\Omega)$. By $r<p^*_s$ we have
\[\frac{N}{ps}(r-p) < r,\]
hence in ${\bf H}$ \ref{h3} we may choose $q<r$ without loss of generality. Therefore, there exists $\tau\in(0,1)$ s.t.\
\[\frac{1}{r} = \frac{1-\tau}{q}+\frac{\tau}{p^*_s}.\]
By the interpolation inequality and the embedding $\w\hookrightarrow L^{p^*_s}(\Omega)$ we have
\[\|u_n\|_r \le \|u_n\|_q^{1-\tau}\|u_n\|_{p^*_s}^\tau \le C\|u_n\|_{s,p}^\tau.\]
By \eqref{cer1} again and \eqref{cer3} we have
\begin{align*}
\|u_n\|_{s,p}^p &\le p\int_\Omega F(x,u_n)\,dx+C \\
&\le C(1+\|u_n\|_r^r) \\
&\le C(1+\|u_n\|_{s,p}^{\tau r}).
\end{align*}
In fact we have $\tau r<p$, indeed by ${\bf H}$ \ref{h3}
\[q > \frac{(r-p)p^*_s}{p^*_s-p} \ \Longrightarrow \ \tau = \frac{(r-q)p^*_s}{(p^*_s-q)r} < \frac{p}{r}.\]
Therefore, $(u_n)$ is bounded in $\w$ as well. Besides, $\Phi'(u_n)\to 0$. By Proposition \ref{bps}, $(u_n)$ has a convergent subsequence, so $\Phi$ satisfies the $(C)$-condition.
\vskip2pt
\noindent
The argument for $\Phi_\pm$ is analogous, with the difference that we test \eqref{cer2} with $u^\pm_n$ and find that $(u_n^\pm)$ is bounded in $L^q(\Omega)$, while $u_n^\mp\to 0$ in $L^q(\Omega)$.
\end{proof}

\noindent
Since $f(\cdot,0)=0$, problem \eqref{dir} obviously admits the trivial solution $u=0$. The next lemma focuses on the nature of such solution as a critical point:

\begin{lemma}\label{zer}
Let ${\bf H}$ hold. Then, $0$ is a local minimizer of both $\Phi$ and $\Phi_\pm$.
\end{lemma}
\begin{proof}
Again we consider $\Phi$. By ${\bf H}$ \ref{h4}, for all $\eps>0$ we can find $\delta>0$ s.t.\ for a.e.\ $x\in\Omega$ and all $|t|\le\delta$
\[|F(x,t)| \le \eps|t|^p.\]
Set
\[\rho = \frac{\delta}{{\rm diam}(\Omega)^s} > 0.\]
Then, for all $u\in\w\cap\cs$ with $\|u\|_{0,s}\le\rho$ we have $\|u\|_\infty\le\delta$. Indeed, for all $x\in\Omega$
\[|u(x)| \le \rho\ds(x) \le \delta.\]
So, for all $u\in\w\cap\cs$, $\|u\|_{0,s}\le\rho$ we have
\begin{align*}
\Phi(u) &\ge \frac{\|u\|_{s,p}^p}{p}-\eps\|u\|_p^p \\
&\ge \Big(\frac{1}{p}-\frac{\eps}{\lambda_1}\Big)\|u\|_{s,p}^p,
\end{align*}
with $\lambda_1>0$ defined as in Proposition \ref{pev}. Choosing $\eps<\lambda_1/p$, we see that $\Phi(u)\ge 0$, hence $0$ is a local minimizer of $\Phi$ in $\cs$. By Proposition \ref{svh}, it is such in $\w$ as well.
\vskip2pt
\noindent
The argument for $\Phi_\pm$ is analogous.
\end{proof}

\noindent
In the following arguments, we will always assume that both $\Phi$ and $\Phi_\pm$ have a {\em strict} local minimum at $0$, otherwise the existence of infinitely many critical points would follow immediately. Anyway, such minumum is not {\em global}. Indeed, due to the $(p-1)$-superlinear growth of $f(x,\cdot)$, the energy functionals are unbounded from below:

\begin{lemma}\label{inf}
Let ${\bf H}$ hold. Then, both
\[\inf_{u\in\w}\Phi(u) = \inf_{u\in\w}\Phi_\pm(u) = -\infty.\]
\end{lemma}
\begin{proof}
Again we deal with $\Phi$. By ${\bf H}$ \ref{h2}, for all $K>0$ we can find $T>0$ s.t.\ for a.e.\ $x\in\Omega$ and all $|t|>T$ there holds
\[F(x,t) \ge K|t|^p.\]
Recalling \eqref{cer3}, we find $C_K>0$ s.t.\ for a.e.\ $x\in\Omega$ and all $t\in\R$
\[F(x,t) \ge K|t|^p-C_K.\]
Let $e_1\in{\rm int}(\cs_+)$ be defined by Proposition \ref{pev}. Then, for all $\tau>0$ we have
\begin{align*}
\Phi(\tau e_1) &\le \frac{\|\tau e_1\|_{s,p}^p}{p}-\int_\Omega\big(K|\tau e_1|^p-C_K\big)\,dx \\
&\le \Big(\frac{\lambda_1}{p}-K\Big)\tau^p+C.
\end{align*}
Choosing $K>\lambda_1/p$, we find
\[\lim_{\tau\to\infty}\Phi(\tau e_1) = -\infty.\]
Therefore, $\Phi$ is unbounded from below in $\w$. The argument for $\Phi_\pm$ is analogous.
\end{proof}

\noindent
The previous lemmas allow us to infer that $\Phi_\pm$ satisfy a mountain pass type geometry around the origin, which along with the $(C)$-condition leads to the following existence result:

\begin{theorem}\label{css}
Let ${\bf H}$ hold. Then, \eqref{dir} has at least two solutions $u_\pm\in\pm{\rm int}(\cs_+)$ s.t.\ $\Phi(u_\pm)>0$.
\end{theorem}
\begin{proof}
We focus on the positive solution and the truncated energy $\Phi_+$. By Lemma \ref{zer} and the subsequent discussion, we may assume that $\Phi_+$ has a strict local minimum at $0$ with $\Phi_+(0)=0$.
\vskip2pt
\noindent
Besides, by Lemma \ref{inf} we can find $\bar u\in\w$ s.t.\
\[\Phi_+(\bar u) < 0.\]
In fact, we can find $\rho\in(0,\|\bar u\|_{s,p})$ s.t.\
\beq\label{css1}
\inf_{\|u\|_{s,p}=\rho}\Phi_+(u) = \eta_\rho > 0.
\eeq
Indeed, since $0$ is a strict local minimizer, then we can find $\rho\in(0,\|\bar u\|_{s,p}/2)$ s.t.\ $\Phi_+(u)>0$ for all $u\in\w$, $0<\|u\|_{s,p}\le 2\rho$. Arguing by contradiction, assume that $\eta_\rho=0$. So, for all $n\in\N$ there exists $v_n\in\partial B_\rho(0)$ s.t.\
\[\Phi_+(v_n) \le \frac{1}{n}.\]
By Ekeland's variational principle \cite[Theorem 5.7]{MMP} there exists $w_n\in\overline{B}_{2\rho}(0)$ s.t.\
\[\Phi_+(w_n) \le \Phi_+(v_n), \ \|w_n-v_n\|_{s,p} \le \frac{1}{\sqrt{n}},\]
and for all $u\in\overline{B}_{2\rho}(0)\setminus\{w_n\}$ there holds
\[\Phi_+(w_n) < \Phi_+(u)+\frac{\|u-w_n\|_{s,p}}{\sqrt{n}}.\]
From the properties of the bounded sequences $(v_n)$, $(w_n)$ we have $\Phi_+(w_n)\to 0$, $\|w_n\|\to\rho$. Also, for all $n\in\N$ and all $\varphi\in\w$
\begin{align*}
\langle\Phi'_+(w_n),\varphi\rangle &= \lim_{t\to 0}\frac{\Phi_+(w_n+t\varphi)-\Phi_+(w_n)}{t} \\
&\ge -\lim_{t\to 0}\frac{\|t\varphi\|_{s,p}}{\sqrt{n}t} = -\frac{\|\varphi\|_{s,p}}{\sqrt{n}},
\end{align*}
while the reverse inequality follows replacing $\varphi$ with $-\varphi$. So we have
\[\|\Phi_+'(w_n)\|_{-s,p'} \le \frac{1}{\sqrt{n}},\]
where we have denoted by $\|\cdot\|_{-s,p'}$ the norm of the dual space $W^{-s,p'}(\Omega)$. By Lemma \ref{cer}, $\Phi_+$ satisfies the $(C)$-condition, hence, passing if necessary to a subsequence, we have $w_n\to w$ in $\w$. From the previous relations we deduce $\|w\|=\rho$, $\Phi_+(w)=0$. This contradicts the choice of $\rho$, and hence it proves \eqref{css1}.
\vskip2pt
\noindent
We will now apply the mountain pass theorem \cite[Theorem 5.40]{MMP}. Set
\[\Gamma = \big\{\gamma\in C([0,1],\w):\,\gamma(0)=0, \ \gamma(1)=\bar u\big\},\]
\[c_+ = \inf_{\gamma\in\Gamma}\max_{t\in[0,1]}\Phi_+(\gamma(t)).\]
Then, $c_+\ge\eta_\rho$ and there exists $u_+\in K_{c_+}(\Phi_+)$. Reasoning as in Subsection \ref{ss21}, we have for all $\varphi\in\w$
\beq\label{css2}
\langle\fpl u_+,\varphi\rangle = \int_\Omega f_+(x,u_+)\varphi\,dx.
\eeq
Choosing $\varphi=-u_+^-\in\w$ in \eqref{css2} and using Proposition \ref{pro}, we have
\begin{align*}
\|u_+^-\|^p &\le \langle\fpl u_+,-u_+^-\rangle \\
&= \int_{\{u_+<0\}}f_+(x,u_+)u_+\,dx = 0.
\end{align*}
So we have $u_+\ge 0$ in $\Omega$. In turn, we can replace $f_+$ with $f$ in \eqref{css2} and see that $u_+$ is a solution of \eqref{dir}. By Propositions \ref{apb}, \ref{reg} we have $u_+\in C^\alpha_s(\overline\Omega)$, in particular $u_+$ is continuous in $\overline\Omega$.
\vskip2pt
\noindent
Using ${\bf H}$ \ref{h4} as in Lemma \ref{zer}, for all $\eps>0$ we can find $\delta\in(0,\|u_+\|_\infty)$ s.t.\ for a.e.\ $x\in\Omega$ and all $t\in[0,\delta]$
\[|f(x,t)| \le \eps|t|^{p-1}.\]
Besides, for all $t\in(\delta,\|u_+\|_\infty]$ we have by ${\bf H}$ \ref{h1}
\begin{align*}
f(x,t) &\ge -c_0(1+t^{r-1}) \\
&\ge -c_0\Big(\frac{t^{p-1}}{\delta^{p-1}}+\|u_+\|_\infty^{r-p}t^{p-1}\Big) \ge -Ct^{p-1}.
\end{align*}
Combining the previous inequalities, we get for a.e.\ $x\in\Omega$ and all $t\in[0,\|u_+\|_\infty]$
\[f(x,t) \ge -Ct^{p-1}.\]
In conclusion we have that $u_+\in\w\cap C^0(\overline\Omega)$ satisfies in a weak sense
\[\begin{cases}
\fpl u_++Cu_+^{p-1} \ge 0 & \text{in $\Omega$} \\
u \ge 0 & \text{in $\R^N$,}
\end{cases}\]
while from $\Phi_+(u_+)=c_+$ it follows that $u_+\neq 0$. By Proposition \ref{smp} (with $g(t)=C(t^+)^{p-1}$) and \eqref{int} we have $u_+\in{\rm int}(\cs_+)$. Finally we point out that
\[\Phi(u_+) = \Phi_+(u_+) = c_+ > 0.\]
The argument for the negative solution $u_-\in -{\rm int}(\cs_+)$ is analogous.
\end{proof}

\begin{remark}\label{art}
Theorem \ref{css} simply represents an adaptation to the nonlinear, nonlocal framework of the ideas of \cite{AR}. No Morse theory is involved so far.
\end{remark}

\section{The third solution}\label{sec4}

\noindent
By Theorem \ref{css} we know that $\Phi$ has at least three critical points, namely $0$ and $u_\pm\in\pm{\rm int}(\cs_+)$ with energies
\[\Phi(u_\pm) = c_\pm > 0 = \Phi(0).\]
In what follows, we will prove the existence of a fourth critical point. Arguing by contradiction, let us assume:
\beq\label{abs}
K(\Phi) = \{0,u_+,u_-\}.
\eeq
We note that, whenever $u\in K(\Phi_+)\setminus\{0\}$, testing $\Phi_+'(u)$ with $-u^-$ and arguing as in Theorem \ref{css} we first find that $u\ge 0$ in $\Omega$, and then that $u\in{\rm int}(\cs_+)$. As a consequence, $u\in K(\Phi)$ and by \eqref{abs} we must have $u=u_+$. A similar argument runs for $\Phi_-$. Thus, under assumption \eqref{abs}, we also have
\[K(\Phi_\pm) = \{0,u_\pm\}.\]
We will now complete the picture by computing all the critical groups of $\Phi$ and $\Phi_\pm$. We begin with critical groups at infinity:

\begin{lemma}\label{cgi}
Let ${\bf H}$, \eqref{abs} hold. Then, for all $k\in\N$
\[C_k(\Phi,\infty) = C_k(\Phi_\pm,\infty) = 0.\]
\end{lemma}
\begin{proof}
We focus on $\Phi$. Since $\Phi$ is sequentially weakly l.s.c.\ and $\w$ is a reflexive space, we have
\[\inf_{u\in\overline{B}_1(0)}\Phi(u) = m > -\infty.\]
Also, from Lemma \ref{inf} we know that $\Phi$ is globally unbounded from below. We claim that for all $v\in\partial B_1(0)$,
\beq\label{cgi1}
\lim_{\tau\to\infty}\Phi(\tau v) = -\infty.
\eeq
Indeed, by $({\bf H}) $ \ref{h2}, for all $M>0$ there exists $T>0$ s.t.\ for a.e.\ $x\in\Omega$ and all $|t|>T$
\[F(x,t) \ge M|t|^p.\]
Besides, by ${\bf H}$ \ref{h1} we have for all $|t|\le T$
\begin{align*}
F(x,t) &\ge -c_0\int_0^{|t|}(1+\tau^{r-1})\,d\tau \\
&\ge -c_0\Big(T+\frac{T^r}{r}\Big).
\end{align*}
So we can find $C_M>0$ s.t.\ for a.e.\ $x\in\Omega$ and all $t\in\R$
\[F(x,t) \ge M|t|^p-C_M.\]
Now fix $v\in\partial B_1(0)$. For all $\tau>0$ we have
\begin{align*}
\Phi(\tau v) &\le \frac{\tau^p\|v\|_{s,p}^p}{p}-\int_\Omega\big(M|\tau v|^p-C_M\big)\,dx \\
&\le \tau^p\Big(\frac{1}{p}-\frac{M}{\lambda_1}\Big)+C,
\end{align*}
where $\lambda_1>0$ is defined as in Proposition \ref{pev}. Choosing $M>\lambda_1/p$, the latter tends to $-\infty$ as $\tau\to\infty$, which proves \eqref{cgi1}.
\vskip2pt
\noindent
Next we prove that there exists $c<m$ s.t.\ for all $u\in\w$, $\Phi(u)=c$ we have
\beq\label{cgi2}
\langle\Phi'(u),u\rangle < 0.
\eeq
Indeed, by ${\bf H}$ \ref{h3} we can find $\beta,T>0$ s.t.\ for a.e.\ $x\in\Omega$ and all $|t|>T$
\[f(x,t)t-pF(x,t) \ge \beta|t|^q.\]
Let $c<m$ (to be determined later) and $\Phi(u)=c$, then using as well ${\bf H}$ \ref{h1} we get
\begin{align*}
\langle\Phi'(u),u\rangle &= \|u\|_{s,p}^p-\int_\Omega f(x,u)u\,dx \\
&= p\Phi(u)-\int_\Omega\big(f(x,u)u-pF(x,u)\big)\,dx \\
&\le pc-\int_{\{|u|>T\}}\beta|u|^q\,dx-\int_{\{|u|\le T\}}\big(f(x,u)u-pF(x,u)\big)\,dx \\
&\le pc-\beta\|u\|_q^q+\int_{\{|u|\le T\}}C(1+|u|^r)\,dx \\
&\le pc+C(1+T^r)|\Omega|,
\end{align*}
with $C>0$ independent of $u$. So choose
\[c < \min\Big\{m,\,-\frac{C(1+T^r)|\Omega|}{p}\Big\},\]
and we plug this into the previous inequality to find \eqref{cgi2}.
\vskip2pt
\noindent
Set for all $v\in\partial B_1(0)$, $\tau\ge 1$
\[\eta_v(\tau) = \Phi(\tau v).\]
Then $\eta_v\in C^1([1,\infty))$ with derivative given for all $\tau\ge 1$ by
\[\eta'_v(\tau) = \langle\Phi'(\tau v),v\rangle.\]
Because of the choice of the constant $c<m$, we have
\[\eta_v(1) = \Phi(v) > c,\]
while by \eqref{cgi1}
\[\lim_{\tau\to\infty}\eta_v(\tau) = -\infty,\]
so there exists $\tau\in(1,\infty)$ s.t.\ $\eta_v(\tau)=c$. In turn, by \eqref{cgi2} this implies
\[\eta'_v(\tau) = \frac{1}{\tau}\langle\Phi'(\tau v),\tau v\rangle < 0.\]
So $\tau>1$ is unique, otherwise we would have at least one $\tau$ s.t.\ $\eta_v(\tau)=c$ and $\eta'_v(\tau)\ge 0$. By the implicit function theorem \cite[Theorem 7.3]{MMP} we can find a mapping $\mu\in C(\partial B_1(0),(1,\infty))$ s.t.\ for all $v\in\partial B_1(0)$ and all $\tau\in(1,\infty)$
\[\Phi(\tau v) \ \begin{cases}
> c & \text{if $\tau<\mu(v)$} \\
= c & \text{if $\tau=\mu(v)$} \\
< c & \text{if $\tau>\mu(v)$.}
\end{cases}\]
In particular we have
\[\Phi^c = \big\{\tau v:\,v\in\partial B_1(0),\,\tau\ge\mu(v)\big\}.\]
Also set
\[E = \big\{\tau v:\,v\in\partial B_1(0),\,\tau\ge 1\big\}.\]
We will now construct some continuous deformations between subsets of $\w$. First set for all $(t,\tau v)\in[0,1]\times E$
\[h(t,\tau v) = \begin{cases}
(1-t)\tau v+t\mu(v)v & \text{if $\tau<\mu(v)$} \\
\tau v & \text{if $\tau\ge\mu(v)$.}
\end{cases}\]
Then, $h:[0,1]\times E\to E$ is a continuous mapping s.t.\ for all $\tau v\in E$
\[h(0,\tau v) = \tau v, \ h(1,\tau v) \in \Phi^c,\]
while for all $t\in[0,1]$ and $\tau v\in\Phi^c$
\[h(t,\tau v) = \tau v.\]
Therefore, $\Phi^c$ is a (strong) deformation retract of $E$. Further, set for all $(t,\tau v)\in[0,1]\times E$
\[\tilde h(t,\tau v) = (1-t)\tau v+tv,\]
so $\tilde h:[0,1]\times E\to E$ is continuous s.t.\ for all $\tau v\in E$
\[\tilde h(0,\tau v) = \tau v, \ \tilde h(1,\tau v)\in\partial B_1(0),\]
while for all $t\in[0,1]$ and $v\in\partial B_1(0)$
\[\tilde h(t,v) = v.\]
So, $\partial B_1(0)$ is as well a (strong) deformation retract of $E$. In view of Proposition \ref{shg} \ref{shg1} \ref{shg3} \ref{shg4} we have for all $k\in\N$
\begin{align*}
H_k(\w,\Phi^c) &= H_k(\w,E) \\
&= H_k(\w,\partial B_1(0)) = 0,
\end{align*}
where in the last passage we have used the contractibility of $\partial B_1(0)$ (this is due to ${\rm dim}(\w)=\infty$). Recall that
\[c < m \le \Phi(0) < \Phi(u_\pm),\]
hence by \eqref{abs} we have
\[c < \inf_{u\in K(\Phi)}\Phi(u).\]
Thus, the previous algebraic relation and the definition of critical groups at infinity (see Subsection \ref{ss22}) imply for all $k\in\N$
\[C_k(\Phi,\infty) = H_k(\w,\Phi^c) = 0.\]
A similar argument holds for $\Phi_\pm$.
\end{proof}

\begin{remark}\label{uni}
We note that the divergence relation \eqref{cgi1} is not uniform with respect to $v$, in general, since $\partial B_1(0)$ is not compact. For this reason $\Phi$, is not anti-coercive in general.
\end{remark}

\noindent
Next we compute the critical points at $0$:

\begin{lemma}\label{cgz}
Let ${\bf H}$, \eqref{abs} hold. Then, for all $k\in\N$
\[C_k(\Phi,0) = C_k(\Phi_\pm,0) = \delta_{k,0}\R.\] 
\end{lemma}
\begin{proof}
Again we consider $\Phi$. We know from Lemma \ref{zer} that $0$ is a local minimizer of $\Phi$, while by \eqref{abs} clearly $0$ is an isolated critical point. In particular, $0$ is a strict local minimizer of $\Phi$. By Proposition \ref{cgp} \ref{cgp1}, then, we have for all $k\in\N$
\[C_k(\Phi,0) = \delta_{k,0}\R.\]
A similar argument stands for $\Phi_\pm$.
\end{proof}

\noindent
Finally we consider $u_\pm$:

\begin{lemma}\label{cgm}
Let ${\bf H}$, \eqref{abs} hold. Then, for all $k\in\N$
\[C_k(\Phi,u_\pm) = \delta_{k,1}\R.\] 
\end{lemma}
\begin{proof}
We focus on $u_+\in{\rm int}(\cs_+)$. By \eqref{abs}, $u_+$ is an isolated critical point for both $\Phi$ and $\Phi_+$. We first claim that for all $k\in\N$
\beq\label{cgm1}
C_k(\Phi,u_+) = C_k(\Phi_+,u_+).
\eeq
Indeed, set for all $\tau\in[0,1]$ and all $u\in\w$
\begin{align*}
\Psi_\tau(u) &= (1-\tau)\Phi(u)+\tau\Phi_+(u) \\
&= \frac{\|u\|_{s,p}^p}{p}-\int_\Omega\big((1-\tau)F(x,u)+\tau F_+(x,u)\big)\,dx.
\end{align*}
By ${\bf H}$ \ref{h1}, for all $\tau\in[0,1]$ the Carath\'eodory mapping $(1-\tau)f+\tau f_+:\Omega\times\R\to\R$ satisfies ${\bf H}_0$ (with constants independent of $\tau$). So, as seen in Subsection \ref{ss21}, $\Psi_\tau\in C^1(\w)$ with derivative given for all $u,\varphi\in\w$ by
\[\langle\Psi'_\tau(u),\varphi\rangle = \langle\fpl u,\varphi\rangle-\int_\Omega\big((1-\tau)f(x,u)+\tau f_+(x,u)\big)\varphi\,dx.\]
Reasoning as in Lemma \ref{cer} we see that $\Psi_\tau$ satsfies the $(C)$-condition. We prove now that $u_+$ is a critical point of $\Psi_\tau$, isolated uniformly with respect to $\tau\in[0,1]$. Indeed, clearly we have for all $\tau\in[0,1]$
\[\Psi'_\tau(u_+) = (1-\tau)\Phi'(u_+)+\tau\Phi'_+(u_+) = 0.\]
Arguing by contradiction, let $(\tau_n)$, $(u_n)$ be sequences in $[0,1]$, $\w\setminus\{u_+\}$ respectively, s.t.\ $u_n\to u_+$ in $\w$ as $n\to\infty$, and for all $n\in\N$ we have $u_n\in K(\Psi_{\tau_n})$, i.e., $u_n\in\w$ solves in a weak sense
\beq\label{cgm2}
\fpl u_n = (1-\tau_n)f(x,u_n)+\tau_n f_+(x,u_n).
\eeq
By ${\bf H}$ \ref{h1} there exists $C>0$ (independent of $n$) s.t.\ for all $n\in\N$, a.e.\ $x\in\Omega$, and all $t\in\R$
\[|(1-\tau_n)f(x,t)+\tau_n f_+(x,t)| \le C(1+|t|^{r-1}).\]
Since the sequence $(u_n)$ is bounded in $\w$, by Proposition \ref{apb} $(u_n)$ is bounded in $L^\infty(\Omega)$ as well. Further, by Proposition \ref{reg} $(u_n)$ is bounded in $C^\alpha_s(\overline\Omega)$. By the compact embedding $C^\alpha_s(\overline\Omega)\hookrightarrow\cs$, passing to a subsequence we have $u_n\to u_+$ in $\cs$.
\vskip2pt
\noindent
Since $u_+\in{\rm int}(\cs_+)$, for all $n\in\N$ big enough we have $u_n\in{\rm int}(\cs_+)$. So, \eqref{cgm2} rephrases as
\[\fpl u_n = f(x,u_n) \ \text{in $\Omega$.}\]
Thus we have
\[u_n \in K(\Phi)\setminus\{0,u_\pm\},\]
against \eqref{abs}.
\vskip2pt
\noindent
Since $u_+$ is uniformly isolated as a critical point of $\Psi_\tau$ for all $\tau\in[0,1]$, then we can apply Proposition \ref{inv} and see that $C_k(\Psi_\tau,u_+)$ is independent of $\tau\in[0,1]$. Noting that $\Psi_0=\Phi$ and $\Psi_1=\Phi_+$, we get \eqref{cgm1}.
\vskip2pt
\noindent
Next we prove that for all $k\in\N$
\beq\label{cgm3}
C_k(\Phi_+,u_+) = \delta_{k,1}\R.
\eeq
Indeed, fix $a,b\in\R$ s.t.\
\[b < 0 < a < \Phi_+(u_+).\]
By Proposition \ref{cgp} \ref{cgp2}, assumption \eqref{abs}, and Lemmas \ref{cgi}, \ref{cgz}, we have for all $k\in\N$ the following algebraic relations:
\[H_k(\w,\Phi^b_+) = C_k(\Phi_+,\infty) = 0,\]
\[H_k(\Phi_+^a,\Phi_+^b) = C_k(\Phi_+,0) = \delta_{k,0}\R,\]
\[H_k(\w,\Phi^a_+) = C_k(\Phi_+,u_+).\]
We apply Proposition \ref{shg} \ref{shg5} with $A=\Phi_+^a$, $B=\Phi_+^b$. With the obvious meaning of the mappings, the following sequence is exact:
\[\ldots \ H_k(\Phi_+^a,\Phi_+^b) \ \xrightarrow{i_*} \ H_k(\w,\Phi_+^b) \ \xrightarrow{j_*} \ H_k(\w,\Phi_+^a) \ \xrightarrow{h_*\circ\partial} \ H_{k-1}(\Phi_+^a,\Phi_+^b) \ \ldots\]
In view of the relations above we have exactness of the following sequence
\[\delta_{k,0}\R \ \xrightarrow{i_*} \ 0 \ \xrightarrow{j_*} \ C_k(\Phi_+,u_+) \ \xrightarrow{h_*\circ\partial} \ \delta_{k-1,0}\R \ \xrightarrow{i_*} \ 0.\]
Then $\partial$ is an isomporhism, i.e., we have
\[C_k(\Phi_+,u_+) = \delta_{k-1,0}\R = \delta_{k,1}\R,\]
which proves \eqref{cgm3}. By \eqref{cgm1} and \eqref{cgm3} we conclude.
\vskip2pt
\noindent
A similar argument, involving the functional $\Phi_-$ in the place of $\Phi_+$, holds for $u_-$.
\end{proof}

\noindent
We are now in a favorable position to prove our main result.
\vskip4pt
\noindent
{\em Proof of Theorem \ref{main}.} We argue under assumption \eqref{abs}. We recall the Poincar\'e-Hopf formula (Proposition \ref{cgp} \ref{cgp3}), which in this case reads as
\[\sum_{k=0}^\infty(-1)^k\big[{\rm dim}(C_k(\Phi,0))+{\rm dim}(C_k(\Phi,u_+))+{\rm dim}(C_k(\Phi,u_-))\big] = \sum_{k=0}^\infty(-1)^k{\rm dim}(C_k(\Phi,\infty)).\]
In addition, from Lemmas \ref{cgi}, \ref{cgz}, and \ref{cgm} we know the critical groups of $\Phi$ at all of its critical points, and at infinity, so we have
\[\sum_{k=0}^\infty(-1)^k(\delta_{k,0}+2\delta_{k,1}) = 0,\]
which reduces to $-1=0$, a contradiction. Thus \eqref{abs} fails, i.e., there exists
\[\tilde u\in K(\Phi)\setminus\{0,u_\pm\}.\]
We already know from Theorem \ref{css} that $u_\pm\in\pm{\rm int}(\cs_+)$ solve \eqref{dir}. Similarly, since $\tilde u\in K(\Phi)$ we apply Propositions \ref{apb} and \ref{reg} to see that $\tilde u\in C^\alpha_s(\overline\Omega)\setminus\{0\}$ and it solves \eqref{dir}. \qed

\begin{remark}\label{sign}
In general, our method bears no sign information about the solution $\tilde u$. Such information can be retrieved in the $(p-1)$-sublinear case, by using the sub-supersolution method introduced in \cite{FI}, but we are not aware of such a result for the superlinear case. We also refer the reader to \cite[Theorem 4.1]{IMS}. Such result, which is formally analogous to Theorem \ref{main}, is proved for the sublinear case but with a completely different underlying variational structure: in \cite{IMS} the constant sign solutions are local minimizer of the energy functional and the third solution is a mountain pass, while here the constant sign solutions are mountain passes and the third is detected via Morse theory.
\end{remark}

\vskip4pt
\noindent
{\small {\bf Acknowledgement.} The first and third authors are members of GNAMPA (Gruppo Nazionale per l'Analisi Matematica, la Probabilit\`a e le loro Applicazioni) of INdAM (Istituto Nazionale di Alta Matematica 'Francesco Severi'). A.\ Iannizzotto is partially supported by the research project {\em Problemi non locali di tipo stazionario ed evolutivo} (GNAMPA, CUP E53C23001670001). This work was partially performed during V.\ Staicu's visit to the University of Cagliari, within the INdAM Visiting Professors Program (May, 2023). V.\ Vespri is supported by the MUR-PRIN 2022 project {\em Detokode - DEsign a governance for the TOKenised ecOnomy in a Decentralised Era}. We would like to thank the anonymous Referee for careful reading of our work.}


\begin{thebibliography}{99}

\bibitem{AR}
{\sc A.\ Ambrosetti, P.H.\ Rabinowitz,}
Dual variational methods in critical point theory and applications,
{\em J. Functional Analysis} {\bf 14} (1973) 349--381.

\bibitem{A}
{\sc V.\ Ambrosio,}
Multiple solutions for a fractional $p$-Laplacian equation with sign-changing potential,
{\em Electron. J. Differential Equations} {\bf 2016} (2016) art.\ 151.

\bibitem{AI}
{\sc V.\ Ambrosio, T.\ Isernia,}
The critical fractional Ambrosetti-Prodi problem,
{\em Rend. Circ. Mat. Palermo} {\bf 71} (2022) 1107--1132.

\bibitem{BGCV}
{\sc L.\ Brasco, D.\ G\'omez-Casto, J.L.\ V\'azquez,}
Characterisation of homogeneous fractional Sobolev space,
{\em Calc. Var. Partial Differential Equations} {\bf 60} (2021) art.\ 60.

\bibitem{CDI}
{\sc F.M.\ Cassanello, F.G.\ D\"uzg\"un, A.\ Iannizzotto,}
H\"older regularity for the fractional $p$-Laplacian, revisited,
preprint (arXiv:2409.19057).

\bibitem{C}
{\sc K.C.\ Chang,}
Infinite dimensional Morse theory and multiple solution problems,
Birkh\"auser, Boston (1993).

\bibitem{CW}
{\sc X.\ Chang, Z.Q.\ Wang,}
Nodal and multiple solutions of nonlinear problems involving the fractional Laplacian,
{\em J. Differential Equations} {\bf 256} (2014) 2965--2992.

\bibitem{CM}
{\sc D.G.\ Costa, C.A.\ Magalh\~aes,}
Variational elliptic problems which are nonquadratic at infinity,
{\em Nonlinear Anal.} {\bf 23} (1994) 1401--1412.

\bibitem{DNPV}
{\sc E.\ Di Nezza, G.\ Palatucci, E.\ Valdinoci},
Hitchhiker's guide to the fractional Sobolev spaces,
{\em Bull. Sci. Math.} {\bf 136} (2012) 521--573.

\bibitem{DI}
{\sc F.G.\ D\"uzg\"un, A.\ Iannizzotto,}
Three nontrivial solutions for nonlinear fractional Laplacian equations,
{\em Adv. Nonlinear Anal.} {\bf 7} (2018) 211--226.

\bibitem{DIV}
{\sc F.G.\ D\"uzg\"un, A.\ Iannizzotto, V.\ Vespri,}
A clustering theorem in fractional Sobolev spaces,
preprint (arXiv:2305.19965).

\bibitem{FMBS}
{\sc A.\ Fiscella, G.\ Molica Bisci, R.\ Servadei,}
Multiplicity results for fractional Laplace problems with critical growth,
{\em Manuscripta Math.} {\bf 155} (2018) 369--388.

\bibitem{FP}
{\sc G.\ Franzina, G.\ Palatucci,}
Fractional $p$-eigenvalues,
{\em Riv. Mat. Univ. Parma} {\bf 5} (2014) 373--386. 

\bibitem{FI}
{\sc S.\ Frassu, A.\ Iannizzotto,}
Extremal constant sign solutions and nodal solutions for the fractional $p$-Laplacian,
{\em J. Math. Anal. Appl.} {\bf 501} (2021) art.\ 124205.

\bibitem{GZZ}
{\sc G.\ Gu, W.\ Zhang, F.\ Zhao,}
Infinitely many sign-changing solutions for a nonlocal problem,
{\em Ann. Mat. Pura Appl.} {\bf 197} (2018) 1429--1444.

\bibitem{GV}
{\sc M.\ Guedda, L.\ Veron,}
Quasilinear elliptic equations involving critical Sobolev exponent,
{\em Nonlinear Anal.} {\bf 13} (1989) 879--902.

\bibitem{ILPS}
{\sc A.\ Iannizzotto, S.\ Liu, K.\ Perera, M.\ Squassina},  
Existence results for fractional $p$-Laplacian problems via Morse theory,
{\em Adv. Calc. Var.} {\bf 9} (2016) 101--125.

\bibitem{IL}
{\sc A.\ Iannizzotto, R.\ Livrea,}
Four solutions for fractional $p$-Laplacian equations with asymmetric reactions,
{\em Mediterr. J. Math.} {\bf 18} (2021) art.\ 220.

\bibitem{IM}
{\sc A.\ Iannizzotto, S.\ Mosconi,}
Fine boundary regularity for the singular fractional $p$-Laplacian,
{\em J. Differential Equations} {\bf 412} (2024) 322--379.

\bibitem{IM1}
{\sc A.\ Iannizzotto, S.\ Mosconi,}
On a doubly sublinear fractional $p$-Laplacian equation,
preprint (arXiv:2409.03616v1).

\bibitem{IMP}
{\sc A.\ Iannizzotto, S.\ Mosconi, N.S.\ Papageorgiou,}
On the logistic equation for the fractional $p$-Laplacian,
{\em Math. Nachr.} {\bf 296} (2023) 1451--1468.

\bibitem{IMS}
{\sc A.\ Iannizzotto, S.\ Mosconi, M.\ Squassina,}
Sobolev versus H\"older minimizers for the degenerate fractional $p$-Laplacian,
{\em Nonlinear Anal.} {\bf 191} (2020) art.\ 111635.

\bibitem{L}
{\sc G.\ Leoni,}
A first course in fractional Sobolev spaces,
American Mathematical Society, Providence (2023).

\bibitem{L1}
{\sc G.M.\ Lieberman,}
Boundary regularity for solutions of degenerate elliptic equations,
{\em Nonlinear Anal.} {\bf 12} (1988) 1203--1219.

\bibitem{MBRS}
{\sc G.\ Molica Bisci, V.D.\ R\u{a}dulescu, R.\ Servadei,}
Variational methods for nonlocal fractional problems.
Cambridge University Press, Cambridge (2016).

\bibitem{MMP}
{\sc D.\ Motreanu, V.V.\ Motreanu, N.S.\ Papageorgiou,}
Topological and variational methods with applications to nonlinear boundary value problems,
Springer, New York (2014).

\bibitem{MP}
{\sc D.\ Mugnai, D.\ Pagliardini,}
Existence and multiplicity results for the fractional Laplacian in bounded domains,
{\em Adv. Calc. Var.} {\bf 10} (2017) 111--124.

\bibitem{P}
{\sc G.\ Palatucci,}
The Dirichlet problem for the $p$-fractional Laplace equation,
{\em Nonlinear Anal.} {\bf 177} (2018) 699--732.

\bibitem{R}
{\sc P.H.\ Rabinowitz,}
Minimax methods in critical point theory with applications to differential equations,
American Mathematical Society, Providence (1986).

\bibitem{ROS}
{\sc X.\ Ros-Oton, J.\ Serra,}
The Dirichlet problem for the fractional Laplacian: regularity up to the boundary,
{\em J. Math. Pures Appl.} {\bf 101} (2014) 275--302.

\bibitem{V}
{\sc J.L.\ V\'azquez,}
A strong maximum principle for some quasilinear elliptic equations,
{\em Appl. Math. Optim.} {\bf 12} (1984) 191--202.

\bibitem{W}
{\sc Z.Q.\ Wang,}
On a superlinear elliptic equation,
{\em Ann. Inst. H. Poincar\'e C Anal. Non Lin\'eaire} {\bf 8} (1991) 43-57.

\bibitem{XZR}
{\sc M.\ Xiang, B.\ Zhang, V.D.\ R\u{a}dulescu,}
Existence of solutions for perturbed fractional $p$-Laplacian equations,
{\em J. Differential Equations} {\bf 260} (2016) 1392--1413.

\bibitem{ZF}
{\sc B.\ Zhang, M.\ Ferrara,}
Multiplicity of solutions for a class of superlinear non-local fractional equations,
{\em Complex Var. Elliptic Equ.} {\bf 60} (2015) 583--595.

\bibitem{ZCC}
{\sc L.\ Zhao, H.\ Cai, Y.\ Chen,}
Multiple nontrivial solutions of superlinear fractional Laplace equations without $(AR)$ condition,
{\em Adv. Nonlinear Anal.} {\bf 12} (2023) art.\ 20220281.

\end{thebibliography}
\end{document}